\documentclass[12pt]{article} 

\usepackage[latin1]{inputenc}
\usepackage[T1]{fontenc}
\usepackage[english]{babel}
\usepackage{amsfonts,amsmath,amssymb,amsthm}
\usepackage{mathrsfs,dsfont} 
\usepackage{graphicx}
\usepackage{bbm}
\usepackage{color}
\usepackage{comment}
\usepackage[textsize=small]{todonotes}       
\usepackage{color}
\usepackage{enumerate}
\usepackage{hyperref}

\newtheorem{theorem}{Theorem}[section]

\newtheorem{definition}[theorem]{Definition}
\newtheorem{lemma}[theorem]{Lemma}

\newtheorem{proposition}[theorem]{Proposition}

\newtheorem{condition}[theorem]{Condition}

\numberwithin{equation}{section}

\newcommand{\be}{\begin{equation}}
\newcommand{\ee}{\end{equation}}
\newcommand{\nn}{\nonumber}

\newcommand{\prob}{\mathbb P}
\newcommand{\expec}{\mathbb E}
\newcommand{\ind}[1]{\mathds 1_{\{#1\}}}

\newcommand{\dint}{{\rm d}}
\newcommand{\bbR}{\mathbb{R}}

\newcommand{\calF}{\mathcal{F}}
\newcommand{\calN}{\mathcal{N}}
\newcommand{\calO}{\mathcal{O}}
\newcommand{\calU}{\mathcal{U}}

\newcommand{\given}{\,|\,}
\DeclareMathOperator*{\argmax}{arg\,max}

\begin{document}

\title{Berry-Esseen bounds in the inhomogeneous Curie-Weiss model with external field}

\author{
Sander Dommers
\footnote{University of Hull, School of Mathematics and Physical Sciences, Cottingham Road, HU6 7RX Hull, United Kingdom. {\tt s.dommers@hull.ac.uk}}
\and
Peter Eichelsbacher
\footnote{Ruhr-Universit\"at Bochum, Fakult\"at f\"ur Mathematik, Universit\"atsstra\ss e 150, 44780 Bochum, Germany. {\tt peter.eichelsbacher@ruhr-uni-bochum.de}}
}

\maketitle

\begin{abstract}
We study the inhomogeneous Curie-Weiss model with external field, where the inhomogeneity is introduced by adding a positive weight to every vertex and letting the interaction strength between two vertices be proportional to the product of their weights. In this model, the sum of the spins obeys a central limit theorem outside the critical line. We derive a Berry-Esseen rate of convergence for this limit theorem using Stein's method for exchangeable pairs. For this, we, amongst others, need to generalize this method to a multidimensional setting with unbounded random variables.
\end{abstract}


\section{Introduction, model and main results}
The inhomogeneous Curie-Weiss model (ICW) was recently introduced in~\cite{GiaGibHofPri16}. In this model, every vertex has an Ising spin attached to it and also has a positive weight. The spins interact with each other, where the (ferromagnetic) interaction strength between two spins is proportional to the product of the weights of the vertices, and the spins also interact with an external field.

This model arose in the study of the annealed Ising model on inhomogeneous random graphs \cite{GiaGibHofPri16}. In the inhomogeneous random graph model, an edge between two vertices is present in the graph with a probability that is proportional to the product of the weights of the vertices. Annealing the Ising model over these random graphs by taking appropriate expectations results in a mean-field type model where spins interact with an average of their neighborhood. When two weights are large, there will be an edge between them more often in the random graph, and therefore the interaction strength in the annealed model will be large as well. Indeed, it can be shown that the resulting interaction is, approximately, also proportional to the product of the weights.

In~\cite{GiaGibHofPri16}, it is proved that in the ICW, and hence also the annealed Ising model on inhomogeneous random graphs, the sum of spins, in the presence of an external field or above the critical temperature, satisfies a central limit theorem. The study of this model continued in~\cite{DomGiaGibHofPri16}, where critical exponents were computed and a non-standard limit theorem was obtained at the critical point, and in~\cite{DomGiaGibHof18}, where large deviations of the sum of spins were studied.

Stein's method for exchangeable pairs was introduced in \cite{Ste86} and is now a popular method to obtain rates of convergence for central  and other limit theorems. Given a random variable $X$, Stein's
method is based on the construction of another variable $X'$ (some coupling) such that the pair
$(X,X')$ is exchangeable, i.e., their joint distribution is symmetric. The approach essentially uses the
elementary fact that if $(X,X')$ is an exchangeable pair, then $\mathbb{E} g(X,X') = 0$ for all antisymmetric
measurable functions $g(x, y)$ such that the expectation exists. A theorem of Stein shows that a measure of proximity of 
$X$ to normality may be provided in terms of
the exchangeable pair, requiring $X'-X$ to be sufficiently small, see \cite[Theorem 1, Lecture III]{Ste86}.
Stein's approach has been successfully applied in many models, see e.g.\ \cite{bookDiaconis}  and references
therein. In \cite{RR}, the range of application was extended by replacing the linear regression property
by a weaker condition. Moreover the method was successfully applied to several mean-field models in statistical mechanics, including the (homogeneous) Curie-Weiss model~\cite{ChaSha11, EicLow10}, the Hopfield model~\cite{EicMar14}, the Curie-Weiss-Potts model~\cite{EicMar15} and $O(N)$ models~\cite{KirMec13,KirNaw16}.

In this paper, we derive a Berry-Esseen rate of convergence for the central limit theorem of the sum of spins in the ICW, i.e., we show that the Kolmogorov distance between the normalized sum of spins and the normal distribution is bounded from above by a constant divided by the square root of the number of vertices. This generalizes the results in~\cite{EicLow10} to the inhomogeneous setting and also to the setting with an external field.

When deriving the so-called {\it regression equation} for the sum of spins, which is the starting point of Stein's method for exchangeable pairs, one sees that not only the sum of spins, but also a weighted sum of spins shows up, where every spin value is multiplied by the weight of its vertex. Hence, one obtains a two-dimensional regression equation. Looking at the joint distribution of the sum of spins and the weighted sum of spins is for example also used to study their large deviations~\cite{DomGiaGibHof18}. Another complication that arises is that the weighted spin sum is not necessarily uniformly bounded.

Multidimensional versions of Stein's method for exchangeable pairs are for example studied in \cite{GesineAdrian}
and \cite{FanRol15}. Stein's method for unbounded exchangeable pairs have for example been studied in~\cite{CheSha12} and~\cite{ShaZha17}. We combine ideas from the latter paper with ideas from~\cite{FanRol15} to derive bounds between marginals of unbounded multidimensional random variables to the standard normal distribution.

The rest of this paper is organized as follows. In the next subsections we formally introduce the ICW, state our main results and provide a short discussion. In Section~\ref{sec-Stein}, we prove the version of Stein's method we need. Finally, in Section~\ref{sec-BEICW}, we use this to prove the Berry-Esseen bound for the ICW.


\subsection{The inhomogeneous Curie-Weiss model}
We now formally introduce the model and present some preliminary results on this model. We write $[n]:=\{1,\ldots,n\}$ and to every vertex $i\in [n]$ we assign a weight $w_i>0$. We need to make some assumptions on the weight sequence $(w_i)_{i\in[n]}$ which are stated below, where we write $W_n=w_I$, with $I\sim Uni[n]$.
\begin{condition}[Weight regularity]\label{cond-WeightReg} There exists a random variable $W$  such that, as $n\rightarrow\infty$,
\begin{enumerate}[(i)]
\item $W_n \stackrel{d}{\longrightarrow} W$,
\item $\mathbb{E}[W_n^2] =\frac{1}{n}\sum_{i\in [n]} w^2_i  \rightarrow \mathbb{E}[W^2]< \infty$,
\item $\mathbb{E}[W_n^3] =\frac{1}{n}\sum_{i\in [n]} w_i^3  \rightarrow \mathbb{E}[W^3]< \infty$.
\end{enumerate} 
Further, we assume that $\mathbb{E}[W]>0$.
\end{condition}

The inhomogeneous Curie-Weiss model is then defined as follows:
\begin{definition}[Inhomogeneous Curie-Weiss model]
Given the weights $(w_i)_{i\in[n]}$, 
the inhomogeneous Curie-Weiss model is defined by the Boltzmann-Gibbs measure which is, for 
any 

\noindent
${\sigma = \{\sigma_i\}_{i\in[n]} \in \{-1,1\}^{n}}$, given by
\be\label{eq-BoltzmannGibbs}
\mu_{n}(\sigma) = \frac{e^{-H_n(\sigma)}}{Z_n},
\ee
where $H_n(\sigma)$ is the Hamiltonian given by
$$
H_n(\sigma)=-\frac{\beta}{2 \ell_n}\biggl(\sum_{i\in[n]}w_i \sigma_{i}\biggr)^2-h\sum_{i \in[n]}\sigma_i,
$$
with $\beta\geq0$ the inverse temperature, $h\in\mathbb{R}$ the external magnetic field and
$$
\ell_n=\sum_{i\in [n]} w_i =n\expec[W_n],
$$
and where $Z_{n}$ is the normalizing partition function, i.e.,
$$
Z_n = \sum_{\sigma \in \{-1,1\}^{n}}e^{-H_n(\sigma)}.
$$
\end{definition}

Note that we retrieve the standard Curie-Weiss model with external field by choosing $w_i\equiv 1$.

The inhomogeneous Curie-Weiss model was obtained in~\cite{GiaGibHofPri16} by annealing the Ising model over inhomogeneous random graphs with these weights. In that case $\beta$ has to be replaced by $\sinh \beta$ and several error terms have to be considered. For simplicity, we here only study the model stated above.

For a given configuration $\sigma$, let $m_n$ be the average spin value, i.e.,
$$
m_n=\frac{1}{n} \sum_{i\in[n]}\sigma_i.
$$
Several properties of $m_n$ under the Boltzmann-Gibbs measure~\eqref{eq-BoltzmannGibbs} have been obtained in~\cite{GiaGibHofPri16}. We summarize the results that are important for this paper below.

In~\cite{GiaGibHofPri16} first of all, it is shown that, for $h\neq0$, the magnetization in the thermodynamic limit equals
\be\label{eq-magnetization}
M(\beta,h) := \lim_{n\to\infty}  \expec[m_n] = \expec\left[\tanh\left(\sqrt{\frac{\beta}{\expec[W]}}W x^* + h\right) \right],
\ee
where $x^* := x^*(\beta,h)$ is equal to the unique solution with the same sign as $h$ of the fixed point equation
\be\label{eq-fixedpoint-x}
x^* = \expec\left[\tanh\left(\sqrt{\frac{\beta}{\expec[W]}}W x^* + h\right)\sqrt{\frac{\beta}{\expec[W]}} W \right].
\ee
When $h\to0$, the model undergoes a phase transition, i.e., there exists a $\beta_c\geq0$, such that the spontaneous magnetization
$$
M(\beta, 0^+):= \lim_{h\searrow0} M(\beta,h) \left\{\begin{array}{ll} =0, \qquad & {\rm for\ }\beta<\beta_c, \\  >0, \qquad & {\rm for\ }\beta>\beta_c.\end{array} \right.
$$
In~\cite{DomGiaGibHofPri16}, it is shown that for $\beta=\beta_c$ we also have that $M(\beta, 0^+)=0$. The critical value is given by
$$
\beta_c = \frac{\expec[W]}{\expec[W^2]}.
$$
We define the uniqueness CLT regime as
\be \label{CLTunique}
\calU = \{(\beta,h) \,:\, \beta\geq0, h\neq0 {\rm\ or\ } 0<\beta<\beta_c, h=0 \}.
\ee
In the uniqueness CLT regime, the fixed point equation \eqref{eq-fixedpoint-x} has a unique solution, and the sum of spins satisfies the central limit theorem, i.e., for $(\beta,h) \in \calU$,
$$
\sqrt{n}\left(m_n- \expec[m_n] \right) \stackrel{d}{\longrightarrow} \calN(0,\chi),
$$
where $\chi$ is the susceptibility given by
\be
\chi := \chi(\beta,h) :=\lim_{n\to\infty} \frac{\partial}{\partial h} \expec[m_n] = \frac{\partial}{\partial h} M(\beta,h).
\ee
This was proved in \cite{GiaGibHofPri16} by analyzing cumulant generating functions. In this paper, we analyze the rate of convergence for this central limit theorem.

We can make the value of the susceptibility more explicit by carrying out the differentiation of the magnetization:
\begin{align*}
\chi(\beta,h) &=\frac{\partial}{\partial h} M(\beta,h) = \frac{\partial}{\partial h} \expec\left[\tanh\left(\sqrt{\frac{\beta}{\expec[W]}}W x^* + h\right) \right] \nn\\
&= \expec\left[\left(1-\tanh^2\left(\sqrt{\frac{\beta}{\expec[W]}}W x^* + h\right)\right)\left(1+\sqrt{\frac{\beta}{\expec[W]}}W \frac{\partial x^*}{\partial h}\right) \right].
\end{align*}
Using the fixed point equation~\eqref{eq-fixedpoint-x},
\begin{align*}
\frac{\partial x^*}{\partial h} &= \frac{\partial }{\partial h}\expec\left[\tanh\left(\sqrt{\frac{\beta}{\expec[W]}}W x^* + h\right)\sqrt{\frac{\beta}{\expec[W]}} W \right] \nn\\
&= \expec\left[\left(1-\tanh^2\left(\sqrt{\frac{\beta}{\expec[W]}}W x^* + h\right)\right)\left(1+\sqrt{\frac{\beta}{\expec[W]}}W \frac{\partial x^*}{\partial h} \right)\sqrt{\frac{\beta}{\expec[W]}} W \right].
\end{align*}
Solving for $\frac{\partial x^*}{\partial h}$ gives
$$
\frac{\partial x^*}{\partial h} = \frac{\sqrt{\frac{\beta}{\expec[W]}} \expec\left[\left(1-\tanh^2\left(\sqrt{\frac{\beta}{\expec[W]}}W x^* + h\right)\right) W \right]}{1-\frac{\beta}{\expec[W]}\expec\left[\left(1-\tanh^2\left(\sqrt{\frac{\beta}{\expec[W]}}W x^* + h\right)\right) W^2 \right]},
$$
and hence
\be\label{eq-susceptibility}
\chi(\beta,h)=1- \expec\left[\tanh^2\left(\sqrt{\frac{\beta}{\expec[W]}}W x^* + h\right) \right]+\frac{\frac{\beta}{\expec[W]} \expec\left[\left(1-\tanh^2\left(\sqrt{\frac{\beta}{\expec[W]}}W x^* + h\right)\right) W \right]^2}{1-\frac{\beta}{\expec[W]}\expec\left[\left(1-\tanh^2\left(\sqrt{\frac{\beta}{\expec[W]}}W x^* + h\right)\right) W^2 \right]}.
\ee
We define the finite size analogues of $M(\beta,h)$ and $\chi(\beta,h)$, given in~\eqref{eq-magnetization} and~\eqref{eq-susceptibility} respectively, as
\begin{equation} \label{Mn}
M_n := M_n(\beta,h) :=  \expec\left[\tanh\left(\sqrt{\frac{\beta}{\expec[W_n]}}W_n x_n^* + h\right) \right],
\end{equation}
and 
\begin{eqnarray}\label{eq-defchin}
\chi_n := \chi_n(\beta,h) &:= &1- \expec\left[\tanh^2\left(\sqrt{\frac{\beta}{\expec[W_n]}}W_n x_n^* + h\right) \right]
\\ \nonumber
&&+\frac{\frac{\beta}{\expec[W_n]} \expec\left[\left(1-\tanh^2\left(\sqrt{\frac{\beta}{\expec[W_n]}}W_n x_n^* + h\right)\right) W_n \right]^2}{1-\frac{\beta}{\expec[W_n]}\expec\left[\left(1-\tanh^2\left(\sqrt{\frac{\beta}{\expec[W_n]}}W_n x_n^* + h\right)\right) W_n^2 \right]}, 
\end{eqnarray}
respectively, where $x_n^* := x_n^*(\beta,h)$ is equal to the unique solution (see Lemma~\ref{lem-globalminGn} below that shows this uniqueness) with the same sign as $h$ of the fixed point equation
\be\label{eq-fixedpoint-xn}
x_n^* = \expec\left[\tanh\left(\sqrt{\frac{\beta}{\expec[W_n]}}W_n x_n^* + h\right)\sqrt{\frac{\beta}{\expec[W_n]}} W_n \right].
\ee

\subsection{Main results}

Let $d_K$ denote the Kolmogorov distance, i.e., for random variables $X$ and $Y$,
$$
d_K(X,Y) := \sup_{z\in \bbR} | \prob(X\leq z) - \prob(Y\leq z)|.
$$ 
Our main result is then as follows.
\begin{theorem}[Berry-Esseen bound for the ICW]\label{thm-berryesseen-magnetization}
Let
\be\label{eq-defXn}
X_n =\sqrt{n} \frac{m_n - M_n}{\sqrt{\chi_n}},
\ee
with $M_n$ and $\chi_n$ defined in \eqref{Mn} and \eqref{eq-defchin} and let $Z\sim \mathcal{N}(0,1)$. Suppose that the weights $(w_i)$ satisfy Condition~\ref{cond-WeightReg}(i)--(iii). Then, for all $(\beta,h) \in \mathcal{U}$, there exists a constant $0<C=C(\beta,h, (w_i)_i)<\infty$, such that
\be\label{eq-berryesseenbound}
d_K(X_n,Z) \leq \frac{C}{\sqrt{n}}.
\ee
Note that the constant depends on the entire weight sequence $(w_i)_{i\geq1}$. Since the quantities are uniformly bounded, one can determine them knowing the entire sequence.
Under Condtition~\ref{cond-WeightReg}(i)--(ii),
$$
d_K(X_n,Z) =o(1).
$$
\end{theorem}
\noindent
Note that we do not normalize $m_n$ using its expectation and variance as was done in~\cite{GiaGibHofPri16}, but instead use explicit quantities for this.

\noindent
We prove this theorem in Section~\ref{sec-BEICW} by using a version of Stein's method to estimate the distance from a standard normal distribution of a one dimensional marginal if one has a $d$-dimensional regression equation for exchangeable pairs. For this, suppose that $X$ and $X'$ are $d$-dimensional random vectors for some $d\geq 1$, and that $(X,X')$ is an exchangeable pair, i.e., their joint distribution is symmetric. We write the vector of differences as $D=X-X'$.

\noindent
We suppose that we have a {\it regression equation} for $(X,X')$ of the form
\be\label{eq-regressionMD}
\expec[ D \given X] = \lambda \Lambda X + \lambda R,
\ee
for some $0<\lambda<1$, invertible matrix $\Lambda$ and vector $R$.

\begin{theorem}[Stein's method]\label{thm-MarginalStein}
Suppose that $(X,X')$ is an exchangeable pair for $d$-dimensional vectors $X$ and $X'$
such that~\eqref{eq-regressionMD} holds. Then, with $Z$ a standard normal random variable and with $X_1$ denoting
the first component of vector $X$, we obtain
\be\label{eq-thmStein}
d_K(X_1,Z) \leq \expec\left[\left|1-\frac{1}{2\lambda}\expec\left[\ell D D_1 \given X\right]\right|\right]+ \frac{1}{\lambda} \expec\left[\bigl|\expec\left[|\ell D| D_1 \given X\right]\bigr|\right]+ \frac{\sqrt{2\pi}}{4} \expec\left[|\ell R|\right],
\ee
where $D_1 = X_1 - X_1'$, $\ell$ is the first row of $\Lambda^{-1}$, i.e., $\ell := e_1^t \Lambda^{-1}$, and $\ell D$ and $\ell R$, respectively, denote the
Euclidean scalar product of the vectors.
\end{theorem}

\noindent
The proof of this theorem can be found in Section~\ref{sec-Stein}.
\subsection{Discussion}
\paragraph{Berry-Esseen bound for the sum of weighted spins} In Section~\ref{sec-MGFs}, we prove that also the sum of weighted spins $\sum_{i\in[n]} w_i\sigma_i$ satisfies the central limit theorem. Berry-Esseen bounds for this limit theorem can be derived in a similar way as is done for the sum of spins, although one needs to assume the convergence of one more moment of $W$ compared to Condition~\ref{cond-WeightReg} because of the extra factor $w_i$. We make some more detailed remarks at the end of the paper.

\paragraph{Limit theorems on the critical line.} For $h=0$ and $\beta>\beta_c$, the solution to the fixed point equation~\eqref{eq-fixedpoint-x} is not unique. We expect that our bounds still hold when one conditions on the magnetization being close to the value that corresponds to appropriate fixed point as was done for example in the Curie-Weiss-Potts model, see~\cite[Theorem 1.5]{EicMar15}.

For $h=0$ and $\beta=\beta_c$ the central limit theorem no longer holds. In~\cite{DomGiaGibHofPri16}, it is shown that one has to rescale the sum of spins with a different power of $n$ to obtain a limit theorem, and that the limit is nonnormal. It would be interesting to generalize our methods also to this case, for example by generalizing the density approach used in~\cite{EicLow10,ChaSha11}. However, when the weight distribution has a sufficiently heavy tail, the limiting distribution is of a form that is not covered anymore by the density approach.

\paragraph{Annealed Ising model on inhomogeneous random graphs} As mentioned, the ICW arose as an approximation for the annealed Ising model on inhomogeneous random graphs. We expect that our results remain true for this model, although if one wants to prove this, extra error terms caused by the approximation with the ICW have to be taken into account.

\paragraph{Quenched Ising model on inhomogeneous random graphs} One can also look at the quenched Ising model on inhomogeneous random graphs, i.e., the Ising model on a fixed realization of the random graph. In~\cite{GiaGibHofPri15}, it is shown that also in this case the central limit theorem holds in the uniqueness CLT regime \eqref{CLTunique}. It would be interesting to also obtain the rate of convergence for this model. This model is not of mean-field type, but spins only interact with their direct neighbors, which makes finding a suitable regression equation more difficult.

\paragraph{Inhomogeneous versions of other mean-field models} It would be interesting to see if results for other mean-field models, such as the ones studied in \cite{EicMar14, EicMar15, KirMec13, KirNaw16}, can be generalized to an inhomogeneous setting. For this, a complete multi-dimensional version of Stein's method for unbounded random variables will have to be derived.

In~\cite{DomKulSch17}, continuous spin models on random graphs were studied in the annealed setting, also resulting in a mean-field approximation. It would be interesting to see if the central limit theorem for the sum of spins can also be proved using our techniques for that model.


\section{Stein's method, proof of Theorem~\ref{thm-MarginalStein}}\label{sec-Stein}

In this section, we prove the bound in~\eqref{eq-thmStein}, by using ideas from~\cite[Theorem~2.2]{ShaZha17} and~\cite{FanRol15}. Similar ideas to obtain one-dimensional CLTs in a multidimensional setting were used in~\cite[Construction 1C]{CheRol10}.
\begin{proof}[Proof of Theorem~\ref{thm-MarginalStein}]
Note that it follows from the regression equation~\eqref{eq-regressionMD} that, for any function $F:\mathbb{R}^d\to\mathbb{R}^d$ such that all expectations below exist,
\begin{align*}
\frac1{2\lambda} &\expec\left[\left(\Lambda^{-1}(X'-X)\right)^t \left(F(X')-F(X)\right)\right] \nn\\
& = \frac1{2\lambda} \expec\left[\left(\Lambda^{-1}(X'-X)\right)^t \left(F(X')+F(X) \right)\right] + \frac{1}{\lambda}\expec\left[\left(\Lambda^{-1}(X-X')\right)^t F(X)\right]\nn\\
& = \frac{1}{\lambda}\expec\left[\left(\Lambda^{-1}\expec[ D \given X] \right)^t F(X)\right]\nn\\
&= \expec\left[X^t F(X)\right] + \expec\left[(\Lambda^{-1}R)^t F(X)\right],
\end{align*}
where we used exchangeability in the second equality and~\eqref{eq-regressionMD} in the last equality. 
In particular, by choosing $F=f e_1$ for some function $f:\mathbb{R}\to\mathbb{R}$ such that all expectations below exist and rewriting,
\be\label{eq-XfX}
\expec\left[X_1 f(X_1)\right] = \frac{1}{2\lambda} \expec\left[\ell D \left(f(X_1)-f(X'_1)\right)\right] - \expec\left[\ell R f(X_1)\right].
\ee
Here $\ell$ denotes the first row of $\Lambda^{-1}$.

\noindent
For $z\in\mathbb{R}$, let $f_z$ be the solution of the Stein equation
\be\label{eq-steineq}
f'_{z}(x)-x f_{z}(x) = \ind{x \leq z} - \Phi(z),
\ee
where $\Phi(z)$ is the distribution function of a standard normal random variable. The background of this equation
reads as follows. A standard Gaussian random variable $Z$ is characterized by the fact that for every absolutely continuous function $f:\bbR \to \bbR$ for which $\expec \big[Zf(Z)\big]<\infty$ it holds that 
 \begin{equation} \label{steinnormal}
 \expec \big[f'(Z)-Zf(Z)\big]=0.
 \end{equation}
This together with the definition of the Kolmogorov-distance is the motivation to study the Stein equation.
If we replace $x$ by a random variable $X$ and take expectations in the Stein equation \eqref{eq-steineq}, we infer that 
$$\expec \big[f_z'(X)-Xf_z(X)\big]= \prob [X \leq z] - \Phi(z).
$$
The curious fact is that the left hand side of the last equation is frequently much simpler to bound than the right hand side and leads to the successfulness of the method. 

As shown in, e.g.,~\cite[Lemma~2.3]{CheGolSha11}, $f_z$ satisfies, for all $x\in\mathbb{R}$,
\be\label{eq-steineqprop1}
|x f_z(x)|\leq1,\qquad  |f'_z(x)|\leq 1, \qquad 0<f_z(x)\leq  \frac{\sqrt{2\pi}}{4},
\ee
and $x f_z(x)$ is an increasing function of $x$.

\noindent
If we take $x=X_1$ in~\eqref{eq-steineq} and take expectations on  both sides, we get, also using~\eqref{eq-XfX},
\begin{align*}
\prob[X_1\leq z] - \Phi(z) &= \expec\left[f'_z(X_1)-X_1 f_z(X_1)\right] \nn\\
&=\expec\left[f'_z(X_1)\right] - \frac{1}{2\lambda} \expec\left[\ell D \left(f_z(X_1)-f_z(X'_1)\right)\right] + \expec\left[\ell R f_z(X_1)\right]\nn\\
&=\expec\left[f'_z(X_1)\left(1-\frac{1}{2\lambda}\ell D D_1\right)\right] +\frac{1}{2\lambda} \expec\left[\ell D \int_{-D_1}^0 f_z'(X_1) \dint t \right]\nn\\
&\qquad - \frac{1}{2\lambda} \expec\left[\ell D \int_{-D_1}^0 f'_z(X_1+t) \dint t\right] + \expec\left[\ell R f_z(X_1)\right]\nn\\
&=\expec\left[f'_z(X_1)\left(1-\frac{1}{2\lambda}\expec\left[\ell D D_1\given X\right]\right)\right] +\frac{1}{2\lambda} \expec\left[\ell D \int_{-D_1}^0 \bigl( f_z'(X_1)-f'_z(X_1+t) \bigr) \dint t \right]\nn\\
&\qquad + \expec\left[\ell R f_z(X_1)\right].
\end{align*}
Hence, we can bound, using~\eqref{eq-steineqprop1},
\begin{align}\label{eq-absPminusPhi}
\left|\prob[X_1\leq z] - \Phi(z)\right| &\leq \expec\left[\left|1-\frac{1}{2\lambda}\expec\left[\ell D D_1 \given X\right]\right|\right] + \frac{1}{2\lambda}\left|\expec\left[\ell D \int_{-D_1}^0 \bigl( f_z'(X_1)-f'_z(X_1+t) \bigr) \dint t \right] \right| \nn\\
&\qquad + \frac{\sqrt{2\pi}}{4} \expec\left[|\ell R|\right].
\end{align}
We use~\eqref{eq-steineq} again to rewrite the second term as:
\begin{align*}
\frac{1}{2\lambda}&\left|\expec\left[\ell D \int_{-D_1}^0 \bigl( f_z'(X_1)-f'_z(X_1+t) \bigr) \dint t \right] \right| \nn\\
& = \frac{1}{2\lambda}\left|\expec\left[\ell D \int_{-D_1}^0 \bigl( X_1f_z(X_1)-(X_1+t)f_z(X_1+t)+\ind{X_1\leq z}-\ind{X_1+t\leq z} \bigr) \dint t \right] \right| \nn\\
&\leq \frac{1}{2\lambda}\left(\left|\expec\left[\ell D \int_{-D_1}^0 \bigl( X_1f_z(X_1)-(X_1+t)f_z(X_1+t) \bigr) \dint t \right] \right|+\left|\expec\left[\ell D \int_{-D_1}^0 \bigl( \ind{X_1\leq z}-\ind{X_1+t\leq z} \bigr) \dint t \right] \right|\right) \nn\\
&=: \frac{1}{2\lambda}\left(\left|I_1 \right|+\left|I_2\right|\right).
\end{align*}
Since $x f_z(x)$ is increasing in $x$,
\begin{align*}
0\leq \int_{-D_1}^0 \bigl( X_1f_z(X_1)-(X_1+t)f_z(X_1+t) \bigr) \dint t &\leq \int_{-D_1}^0 \bigl( X_1f_z(X_1)-(X_1-D_1)f_z(X_1-D_1) \bigr) \dint t \nn\\
&=D_1 \left(X_1f_z(X_1)-X_1' f_z(X_1')\right).
\end{align*}
Hence,
\begin{align*}
I_1 &= \expec\left[\ell D\left(\ind{\ell D<0}+\ind{\ell D>0}\right) \int_{-D_1}^0 \bigl( X_1f_z(X_1)-(X_1+t)f_z(X_1+t) \bigr) \dint t \right] \nn\\
&\leq \expec\left[\ell D\ind{\ell D>0} \int_{-D_1}^0 \bigl( X_1f_z(X_1)-(X_1+t)f_z(X_1+t) \bigr) \dint t \right] \nn\\
&\leq \expec\left[|\ell D| \ind{\ell D>0} D_1 \left(X_1f_z(X_1)-X_1' f_z(X_1')\right) \right] \nn\\
&= \expec\left[|\ell D| \left(\ind{\ell D<0}+\ind{\ell D>0}\right) D_1 X_1f_z(X_1) \right] \nn\\
&= \expec\left[\expec\left[|\ell D| D_1 \given X\right] X_1f_z(X_1) \right] \nn\\
&\leq \expec\left[\left|\expec\left[|\ell D| D_1 \given X\right] \right| \right],
\end{align*}
where we used that it follows from exchangeability that
$$
\expec\left[|\ell D| \ind{\ell D>0} D_1 \left(X_1' f_z(X_1')\right) \right] =-\expec\left[|\ell D| \ind{\ell D<0} D_1 \left(X_1 f_z(X_1)\right) \right],
$$
and that $|X_1f_z(X_1)|\leq1$. Similarly,
\begin{align*}
I_1 &\geq \expec\left[\ell D\ind{\ell D<0} \int_{-D_1}^0 \bigl( X_1f_z(X_1)-(X_1+t)f_z(X_1+t) \bigr) \dint t \right] \nn\\
&\geq -\expec\left[|\ell D| \ind{\ell D<0} D_1 \left(X_1f_z(X_1)-X_1' f_z(X_1')\right) \right] \nn\\
&\geq -\expec\left[\left|\expec\left[|\ell D| D_1 \given X\right] \right| \right],
\end{align*}
Combining the upper and lower bound on $I_1$ gives,
\be\label{eq-bdI1}
|I_1| \leq \expec\left[\left|\expec\left[|\ell D| D_1 \given X\right] \right| \right].
\ee

\noindent
We can show in a similar way, using that $\ind{x \leq z}$ is non-increasing in $x$, that also
\be\label{eq-bdI2}
|I_2| \leq \expec\left[\left|\expec\left[|\ell D| D_1 \given X\right] \right| \right].
\ee
Combining~\eqref{eq-absPminusPhi},~\eqref{eq-bdI1} and~\eqref{eq-bdI2} proves the theorem.
\end{proof}


\section{Berry-Esseen bound for the ICW, proof of Theorem~\ref{thm-berryesseen-magnetization}}\label{sec-BEICW}
We now use Theorem~\ref{thm-MarginalStein} to prove the Berry-Esseen bound in~\eqref{eq-berryesseenbound}. First, we define our exchangeable pairs and derive the regression equation in Section~\ref{sec-regression}. Then we prove that the central limit theorem holds for the weighted spin sum in Section~\ref{sec-MGFs} and in particular also show that certain moments converge to that of the normal distribution. Finally, in Section~\ref{sec-errorterms}, we bound all terms of~\eqref{eq-thmStein} to prove Theorem~\ref{thm-berryesseen-magnetization}.
\subsection{Exchangeable pairs and regression equation}\label{sec-regression}
We let $X_n$ be as in~\eqref{eq-defXn}.
Let $I \sim Uni[n]$ and let $\sigma'_I$ be drawn from the conditional distribution given $(\sigma_j)_{j\neq I}$. Define $X'_n$ as
$$
X'_n = X_n - \frac{1}{\sqrt{n}}\frac{\sigma_I- \sigma'_I}{\sqrt{\chi_n}}.
$$
Then, $(X_n, X'_n)$ indeed is an exchangeable pair. 

\noindent
Let
$$
\tilde{M}_n := \tilde{M}_n(\beta,h) := \sqrt{\frac{\expec[W_n]}{\beta}} x_n^*,
$$
with $x_n^*$ given in \eqref{eq-fixedpoint-xn}. Let
$$
\tilde{m}_n = \frac1n \sum_{j\in[n]} w_j \sigma_j.
$$
Let us define the constant
\begin{equation} \label{sigma}
\sigma^2(x_n^*, \beta,h) := \frac{1}{1-\frac{\beta}{\expec[W_n]}\expec\left[\left(1-\tanh^2\left(\sqrt{\frac{\beta}{\expec[W_n]}}W_n x_n^* + h\right)\right)W_n^2\right]},
\end{equation}
and let $\tilde{\chi}_n$ be the constant given by
\begin{align}\label{eq-choicetildechi}
\tilde{\chi}_n := \tilde{\chi}_n(\beta,h) &:= \sigma^2(x_n^*, \beta,h) \expec\left[\left(1-\tanh^2\left(\sqrt{\frac{\beta}{\expec[W_n]}}W_n x_n^* + h\right)\right)W_n^2\right] \nn\\
&= \frac{\expec\left[\left(1-\tanh^2\left(\sqrt{\frac{\beta}{\expec[W_n]}}W_n x_n^* + h\right)\right)W_n^2\right]}{1-\frac{\beta}{\expec[W_n]}\expec\left[\left(1-\tanh^2\left(\sqrt{\frac{\beta}{\expec[W_n]}}W_n x_n^* + h\right)\right)W_n^2\right]}.
\end{align} 
\noindent
Now define $\tilde{X}_n$ as
\be\label{eq-deftildeXn}
\tilde{X}_n =\sqrt{n} \frac{\tilde{m}_n - \tilde{M}_n}{\sqrt{\tilde{\chi}_n}},
\ee
and
$$
\tilde{X}'_n = \tilde{X}_n - \frac{1}{\sqrt{n}}\frac{w_I(\sigma_I- \sigma'_I)}{\sqrt{\tilde{\chi}_n}},
$$
so that also $(\tilde{X}_n, \tilde{X}'_n)$ is an exchangeable pair. 

\noindent
From now on, we write $X=(X_n,\tilde{X}_n)^t$ and $X'=(X'_n,\tilde{X}'_n)^t$.

Denote by $\calF_n$ the sigma algebra generated by $(\sigma_i)_{i\in[n]}$ and by $\calF_n^i$ the sigma algebra generated by $(\sigma_j)_{j\in[n], j \neq i}$. 
Also define
\begin{equation} \label{mni}
\tilde{m}_n^i = \frac1n \sum_{j\in[n]: j\neq i} w_j \sigma_j.
\end{equation}
Then, we can compute
\begin{align*}
\mu_n(\sigma_i \,|\, \calF_n^i) &= \frac{\exp\left(\frac{\beta}{2\ell_n} \left(\sum_{j\in[n]: j\neq i} w_j\sigma_j + w_i\sigma_i\right)^2+h\sum_{j\in[n]}\sigma_j\right)}{\sum_{\sigma'_i\in\{-1,1\}}\exp\left(\frac{\beta}{2\ell_n} \left(\sum_{j\in[n]: j\neq i} w_j\sigma_j + w_i\sigma'_i\right)^2+h\sum_{j\in[n]:j\neq i}\sigma_j+h\sigma'_i \right)} \nn\\
&= \frac{\exp\left(\frac{\beta}{\ell_n} w_i \sigma_i \sum_{j\in[n]: j\neq i} w_j\sigma_j + h\sigma_i\right)}{\exp\left(\frac{\beta}{\ell_n} w_i \sum_{j\in[n]: j\neq i} w_j\sigma_j + h\right)+\exp\left(-\left(\frac{\beta}{\ell_n} w_i \sum_{j\in[n]: j\neq i} w_j\sigma_j + h\right)\right)},
\end{align*}
and hence,
\be\label{eq-sigmaprimegivenF}
\expec[\sigma'_i \,|\, \calF_n] =\expec[\sigma_i \,|\, \calF_n^i]= \tanh\left(\frac{\beta}{\ell_n} w_i \sum_{j\in[n]: j\neq i} w_j\sigma_j + h\right) = \tanh\left(\frac{\beta w_i}{\expec[W_n]} \tilde{m}_n^i+h\right).
\ee
We obtain the following regression equation.

\begin{lemma}\label{lem-regression} Let us define
\be\label{eq-deGn}
G_n(x; s) = \frac{x^2}{2} - \expec\biggl[\log \cosh \biggl(\sqrt{\frac{\beta}{\expec[W_n]}}W_n (x+s)+h\biggr)\biggr],
\ee
and $G_n(x) = G_n(x;0)$. Then for $(\beta,h)\in\calU$ we obtain that
\be\label{eq-regressionICW}
\expec\left[\begin{pmatrix}X_n \\ \tilde{X}_n \end{pmatrix}-\begin{pmatrix} X'_n \\ \tilde{X}'_n \end{pmatrix}  \,|\, \calF_n\right] = \lambda \begin{pmatrix}1 & -c \\ 0 & 1/\sigma^2(x_n^*,\beta,h) \end{pmatrix} \begin{pmatrix}X_n \\ \tilde{X}_n \end{pmatrix} +\lambda \begin{pmatrix} R_1 +R_2 \\ \tilde{R_1}+ \tilde{R_2}\end{pmatrix},
\ee
where $\sigma^2(x_n^*,\beta,h)$ is given in \eqref{sigma}, and 
\be\label{eq-deflambdasigma}
\lambda=1/n, \qquad \sigma^2(x_n^*,\beta,h)= \frac{1}{G''_n(x^*_n)},
\ee
(the latter equality follows from \eqref{eq-deGn}, see below \eqref{eq-Gn2})
\be\label{eq-defc}
c=\frac{\sqrt{\tilde{\chi}_n}}{\sqrt{\chi_n}} \frac{\beta}{\expec[W_n]} \expec\left[ \left(1-\tanh^2\left(\sqrt{\frac{\beta}{\expec[W_n]}} W_n x_n^*+h\right)\right) W_n \right],
\ee
and the error terms are given by
\begin{align}
\label{eq-defR1}
R_1&=  \frac{\sqrt{n}}{\sqrt{\chi_n}}\frac{1}{n}\sum_{i\in[n]} \left( \tanh\left(\frac{\beta w_i}{\expec[W_n]} \tilde{m}_n+h\right) - \tanh\left(\frac{\beta w_i}{\expec[W_n]} \tilde{m}_n^i+h\right) \right), \\
\label{eq-defR2}
R_2&=\frac{\sqrt{n}}{\sqrt{\chi_n}} \frac{1}{n} \sum_{i\in[n]}\left( \tanh\left(\sqrt{\frac{\beta}{\expec[W_n]}} w_i x_n^*+h\right) - \tanh\left(\frac{\beta w_i}{\expec[W_n]} \tilde{m}_n+h\right)  \right)+c \tilde{X}_n, \\
\label{eq-deftildeR1}
\tilde{R}_1 &=   \frac{\sqrt{n}}{\sqrt{\tilde{\chi}_n}} \frac{1}{n}\sum_{i\in[n]}w_i \left( \tanh\left(\frac{\beta w_i}{\expec[W_n]} \tilde{m}_n+h\right) - \tanh\left(\frac{\beta w_i}{\expec[W_n]} \tilde{m}_n^i+h\right) \right), \\
\label{eq-deftildeR2}
\tilde{R}_2 &= \frac{\sqrt{n}}{\sqrt{\tilde{\chi}_n}} \sqrt{\frac{\expec[W_n]}{\beta}} G'_n\left(\sqrt{\frac{\beta}{\expec[W_n]}} \tilde{m}_n\right) - \frac{1}{\sigma^2(x_n^*,\beta,h)}\tilde{X}_n.
\end{align}
\end{lemma}

\begin{proof}
We start by computing $\expec[\tilde{X}_n-\tilde{X}_n' \,|\, \calF_n] $. For this, note that 
\be\label{eq-Xn-Xnprime}
\tilde{X}_n-\tilde{X}_n' =  \frac{1}{\sqrt{n}}\frac{w_I(\sigma_I- \sigma'_I)}{\sqrt{\tilde{\chi}_n}}.
\ee
Hence,
\begin{align*}
\expec[\tilde{X}_n-\tilde{X}_n'  \,|\, \calF_n] & =\frac{1}{\sqrt{n}\sqrt{\tilde{\chi}_n}} \frac{1}{n} \sum_{i\in[n]} w_i \expec[\sigma_i-\sigma'_i \,|\, \calF_n] \nn\\
&=\frac{1}{\sqrt{n}\sqrt{\tilde{\chi}_n}} \tilde{m}_n -\frac{1}{\sqrt{n}\sqrt{\tilde{\chi}_n}} \frac{1}{n}\sum_{i\in[n]}w_i \expec[\sigma'_i \,|\, \calF_n]\nn\\
&= \frac{1}{\sqrt{n}\sqrt{\tilde{\chi}_n}} \tilde{m}_n - \frac{1}{\sqrt{n}\sqrt{\tilde{\chi}_n}} \frac{1}{n}\sum_{i\in[n]}w_i  \tanh\left(\frac{\beta w_i}{\expec[W_n]} \tilde{m}_n^i+h\right) \nn\\
& = \frac{1}{\sqrt{n}\sqrt{\tilde{\chi}_n}} \tilde{m}_n -\frac{1}{\sqrt{n}\sqrt{\tilde{\chi}_n}} \frac{1}{n}\sum_{i\in[n]}w_i  \tanh\left(\frac{\beta w_i}{\expec[W_n]} \tilde{m}_n+h\right) +\lambda \tilde{R}_1,
\end{align*}
where $\tilde{R}_1$ is given in~\eqref{eq-deftildeR1}. Observe that it follows immediately from the definition of $G_n$ in~\eqref{eq-deGn} that
\begin{align}\label{eq-der-Gn}
G_n'(x) &= x - \expec\left[\tanh\left(\sqrt{\frac{\beta}{\expec[W_n]}}W_n x + h\right) \sqrt{\frac{\beta}{\expec[W_n]}}W_n \right] \nn\\
&= x -\frac{1}{n} \sum_{i\in[n]} \tanh\left(\sqrt{\frac{\beta}{\expec[W_n]}}w_i x + h\right) \sqrt{\frac{\beta}{\expec[W_n]}}w_i,
\end{align}
and hence, with $x=\sqrt{\frac{\beta}{\expec[W_n]}}\tilde{m}_n$,
\be\label{eq-regressiontildeX}
\expec[\tilde{X}_n-\tilde{X}_n' \,|\, \calF_n] -\lambda \tilde{R}_1 = \frac{1}{\sqrt{n}\sqrt{\tilde{\chi}_n}} \sqrt{\frac{\expec[W_n]}{\beta}} G'_n\left(\sqrt{\frac{\beta}{\expec[W_n]}} \tilde{m}_n\right)= \frac{\lambda}{\sigma^2(x_n^*,\beta,h)} \tilde{X}_n +\lambda\tilde{R}_2,
\ee
where $\lambda$ and $\sigma^2(x_n^*,\beta,h)$ are given in~\eqref{eq-deflambdasigma} and $\tilde{R}_2$ in~\eqref{eq-deftildeR2}. Note that $\frac{\lambda}{\sigma^2(x_n^*,\beta,h)} \tilde{X}_n$ is equal to the first order Taylor expansion of $G_n'$ around $x_n^*$ and  therefore it is to be expected that $\lambda\tilde{R}_2$, which is the error made in doing so, is small.

Similarly,
\begin{align*}
\expec[X_n&-X_n' \,|\, \calF_n]  = \frac{1}{\sqrt{n}\sqrt{\chi_n}} m_n - \frac{1}{\sqrt{n}\sqrt{\chi_n}} \frac{1}{n}\sum_{i\in[n]}  \tanh\left(\frac{\beta w_i}{\expec[W_n]} \tilde{m}_n^i+h\right) \\
&= \lambda X_n + \frac{1}{\sqrt{n}\sqrt{\chi_n}} \frac{1}{n} \sum_{i\in[n]}\left( \tanh\left(\sqrt{\frac{\beta}{\expec[W_n]}} w_i x_n^*+h\right) - \tanh\left(\frac{\beta w_i}{\expec[W_n]} \tilde{m}_n+h\right)  \right)  +\lambda R_1. \nn
\end{align*}
Using a Taylor expansion, we see that $\tanh(x) \approx \tanh(a)+(1-\tanh^2(a))(x-a)$ for $x$ close to $a$. Hence,
\begin{align}\label{eq-TaylorR2}
\frac{1}{\sqrt{n}\sqrt{\chi_n}} \frac{1}{n} &\sum_{i\in[n]}\left( \tanh\left(\sqrt{\frac{\beta}{\expec[W_n]}} w_i x_n^*+h\right) - \tanh\left(\frac{\beta w_i}{\expec[W_n]} \tilde{m}_n+h\right)  \right) \nn\\
&\approx- \frac{1}{\sqrt{n}\sqrt{\chi_n}} \frac{1}{n} \sum_{i\in[n]} \left(1-\tanh^2\left(\sqrt{\frac{\beta}{\expec[W_n]}} w_i x_n^*+h\right)\right) w_i \frac{\beta}{\expec[W_n]} \left(\tilde{m}_n-\sqrt{\frac{\expec[W_n]}{\beta}}x_n^* \right) \nn\\
&= - \frac{1}{\sqrt{n}\sqrt{\chi_n}} \expec\left[ \left(1-\tanh^2\left(\sqrt{\frac{\beta}{\expec[W_n]}} W_n x_n^*+h\right)\right) W_n \right] \frac{\beta}{\expec[W_n]} \frac{\sqrt{\tilde{\chi}_n}}{\sqrt{n}}\tilde{X}_n \nn\\
&= -\lambda c \tilde{X}_n,
\end{align}
where $c$ is defined in~\eqref{eq-defc}.
Hence, we write
\be\label{eq-regressionX}
\expec[X_n-X_n' \,|\, \calF_n]  = \lambda (X_n - c \tilde{X}_n) + \lambda (R_1+R_2),
\ee
where $R_2$ is given in~\eqref{eq-defR2}.

\noindent
The lemma follows by combining~\eqref{eq-regressiontildeX} and~\eqref{eq-regressionX}.
\end{proof}


\subsection{Central limit theorem for weighted spin sums}\label{sec-MGFs}
In this section, we show that the weighted sum of spins $\sum_{i\in[n]} w_i \sigma_i$ obeys a central limit theorem. We do this by showing that, if we normalize the sum properly, the moment generating function converges to that of a normal distribution. This implies that the normalized sum converges to a normal in distribution and, more importantly for us, that also all moments converge to that of this normal. We also investigate sums of differently weighted spins.

We use the methods to prove the convergence of pressure of the inhomogeneous Curie-Weiss model in~\cite[Sec.~2.1]{GiaGibHofPri16} to prove that certain cumulant generating functions converge, and the methods to prove the CLT for the spin sum in~\cite[Sec.~2.2]{GiaGibHofPri16}, of which the details can be found in the proof of the CLT for the quenched Ising model on random graphs in~\cite[Sec.~2.3]{GiaGibHofPri15}.

\begin{lemma}\label{lem-cmf-weigthedsum}
Define the cumulant generating function 
$$
c_n(s) = \frac1n \log \expec\biggl[\exp\biggl(s \sqrt{\frac{\beta}{\expec[W_n]}}\sum_{i\in[n]} w_i\sigma_i\biggr)\biggr].
$$
Then, for any constant $a$,
$$
c_n(s) = \frac1n \log \frac{\int_{-\infty}^\infty e^{-n G_n\left(\frac{x}{\sqrt{n}}+a;s\right)} \dint x}{\int_{-\infty}^\infty e^{-n G_n\left(\frac{x}{\sqrt{n}}+a\right)} \dint x},
$$
where $G_n(x;s)$ is defined in \eqref{eq-deGn} with $G_n(x) := G_n(x;0)$.
\end{lemma}

\begin{proof}
Note that
\begin{align}\label{eq-cmf-weigthedsum}
\expec\left[e^{s \sqrt{\frac{\beta}{\expec[W_n]}}\sum_{i\in[n]} w_i\sigma_i}\right] &= \frac{\sum_{\sigma\in\{-1,1\}^n}e^{s \sqrt{\frac{\beta}{\expec[W_n]}}\sum_{i\in[n]} w_i\sigma_i}e^{\frac{\beta}{2 n\expec[W_n]}\left(\sum_{i\in[n]}w_i \sigma_{i}\right)^2+h\sum_{i \in[n]}\sigma_i} }{\sum_{\sigma\in\{-1,1\}^n}e^{\frac{\beta}{2 n\expec[W_n]}\left(\sum_{i\in[n]}w_i \sigma_{i}\right)^2+h\sum_{i \in[n]}\sigma_i}}\nn\\
&= \frac{\sum_{\sigma\in\{-1,1\}^n}e^{\frac{\beta}{2 n\expec[W_n]}\left(\sum_{i\in[n]}w_i \sigma_{i}\right)^2+\sum_{i \in[n]}(s\sqrt{\frac{\beta}{\expec[W_n]}}w_i+h)\sigma_i} }{\sum_{\sigma\in\{-1,1\}^n}e^{\frac{\beta}{2 n\expec[W_n]}\left(\sum_{i\in[n]}w_i \sigma_{i}\right)^2+h\sum_{i \in[n]}\sigma_i}}.
\end{align}
Hence, we can interpret the numerator as an inhomogeneous Curie-Weiss model, where also the field is inhomogeneous, i.e., the field at vertex $i$ is given by $s\sqrt{\frac{\beta}{\expec[W_n]}}w_i+h$. We can use the Hubbard-Stratonovich transform $e^{\frac{t^2}{2}}=\expec[e^{tZ}]$, where $Z$ is a standard normal random variable, to rewrite the numerator of~\eqref{eq-cmf-weigthedsum} as
\begin{align*}
\sum_{\sigma\in\{-1,1\}^n}&\expec\Bigl[e^{\sqrt{\frac{\beta}{n\expec[W_n]}}\left(\sum_{i\in[n]}w_i \sigma_{i}\right)Z}\Bigr]e^{\sum_{i \in[n]}(s\sqrt{\frac{\beta}{\expec[W_n]}}w_i+h)\sigma_i} 
= 2^n \expec\Bigl[e^{\sum_{i\in[n]}\log\cosh\left(\sqrt{\frac{\beta}{\expec[W_n]}}w_i \left(\frac{Z}{\sqrt{n}}+s\right)+h\right)} \Bigr] \nn\\
&= \frac{2^n}{\sqrt{2\pi}} \int_{-\infty}^{\infty} e^{\sum_{i\in[n]}\log\cosh\left(\sqrt{\frac{\beta}{\expec[W_n]}}w_i \left(\frac{z}{\sqrt{n}}+s\right)+h\right)}e^{-\frac{z^2}{2}} \dint z\nn\\
&= \frac{2^n}{\sqrt{2\pi}} \int_{-\infty}^{\infty} e^{-n\left[\frac12\left(\frac{z}{\sqrt{n}}\right)^2-\expec\left[\log\cosh\left(\sqrt{\frac{\beta}{\expec[W_n]}} W_n \left(\frac{z}{\sqrt{n}}+s\right)+h\right)\right]\right]} \dint z \nn\\
&=\frac{2^n}{\sqrt{2\pi}} \int_{-\infty}^{\infty} e^{-nG_n\left(\frac{z}{\sqrt{n}};s\right)} \dint z=\frac{2^n}{\sqrt{2\pi}} \int_{-\infty}^{\infty} e^{-nG_n\left(\frac{x}{\sqrt{n}}+a;s\right)} \dint x,
\end{align*}
where we used the change of variables $x=z-\sqrt{n}a$ in the last equality.
The same computation can be done for the denominator of~\eqref{eq-cmf-weigthedsum}, by setting $s=0$.
\end{proof}

An important role is played by the global minimum of $G_n(x)$. In~\cite{GiaGibHofPri16}, it is shown that for $(\beta,h)\in\calU$ in the limit $n\to\infty$ the global minimizer is given by the unique solution with the same sign as $h$ of the fixed point equation~\eqref{eq-fixedpoint-x}.
We give a characterization of the global minimizer for finite $n$ in the next lemma.

\begin{lemma}\label{lem-globalminGn}
Suppose that Condition~\ref{cond-WeightReg}(i)--(ii) holds and that $n$ is large enough. Then, for $0\leq\beta <\beta_c$
and $h=0$, the global minimizer of $G_n(x)$ is given by $x_n^*=0$. For $\beta\geq 0, h\neq 0$, the global minimizer of $G_n(x)$ is given by the unique fixed point $x_n^*$ with the same sign as $h$ of the fixed point equation \eqref{eq-fixedpoint-xn}.
Furthermore, for all $(\beta,h) \in \calU$, 
$$
G_n''(x_n^*)>0.
$$
\end{lemma}

\begin{proof}
Note that $G_n$ is continuous and 
$$
G_n'(x) = x - \expec\left[\tanh\left(\sqrt{\frac{\beta}{\expec[W_n]}}W_n x + h\right)\sqrt{\frac{\beta}{\expec[W_n]}} W_n \right],
$$
and hence the global minimizer has to satisfy~\eqref{eq-fixedpoint-xn}. We can also compute the second derivative:
\be\label{eq-Gn2}
G_n''(x)=1-\expec\left[\left(1-\tanh^2\left(\sqrt{\frac{\beta}{\expec[W_n]}}W_n x + h\right)\right)\frac{\beta}{\expec[W_n]} W_n^2 \right].
\ee
For $x\neq0$, it holds that $0<\tanh^2(x)\leq1$, and hence we can bound
$$
G_n''(x) > 1-\beta \frac{\expec[W_n^2]}{\expec[W_n]}.
$$
Therefore $G_n''(x)>0$ for $x\neq0$ and $\beta < \beta_c$ and $n$ large enough by Condition~\ref{cond-WeightReg}(ii), i.e., $G_n$ is strictly convex. Furthermore, using $\log\cosh x \leq |x|$
\be\label{eq-limG-nxtoinfty}
\lim_{|x|\to\infty} G_n(x) \geq \lim_{|x|\to\infty} \frac{x^2}{2} - \expec\left[\left| \sqrt{\frac{\beta}{\expec[W_n]}}W_n x+h \right|\right]=\infty.
\ee
Therefore, for $\beta<\beta_c$, $G_n(x)$ has a unique local minimum, which must be a global minimum. For $h=0$ clearly $x_n^*=0$ is a fixed point proving the first statement.

Now suppose that $h>0$. Define
$$
H_n(x) = \expec\left[\tanh\left(\sqrt{\frac{\beta}{\expec[W_n]}}W_n x + h\right)\sqrt{\frac{\beta}{\expec[W_n]}} W_n \right].
$$
Then $H_n$ is continuous and 
$$
H_n'(x)= \expec\left[\left(1-\tanh^2\left(\sqrt{\frac{\beta}{\expec[W_n]}}W_n x + h\right)\right)\frac{\beta}{\expec[W_n]} W_n^2 \right],
$$
and
$$
H_n''(x)=-2  \expec\left[\tanh\left(\sqrt{\frac{\beta}{\expec[W_n]}}W_n x + h\right)\left(1-\tanh^2\left(\sqrt{\frac{\beta}{\expec[W_n]}}W_n x + h\right)\right)\left(\frac{\beta}{\expec[W_n]}\right)^{3/2} W_n^3\right].
$$
Since $W_n$ is a positive random variable, we conclude that $H_n$ is concave for $x\geq0$. Since $H_n(0)>0$ and $H_n$ is bounded, there is a unique positive solution to $x=H_n(x)$, call this solution $x_n^*$. Since $G_n'(0)<0$ it follows from~\eqref{eq-limG-nxtoinfty} that $x_n^*$ is a local minimizer.
For any solution $x^-<0$ of~\eqref{eq-fixedpoint-xn},we have that
$$
G_n(x^-) > G_n(-x^-) \geq G_n(x_n^*),
$$
since $x_n^*$ is the unique positive local minimizer. Hence, $x_n^*$ is also the unique global minimizer.

\noindent
The proof for $h<0$ is similar.

\noindent
Since $x_n^*$ is the unique global minimizer, we must have that $G_n''(x) \geq 0$, so it only remains to show that this inequality is strict for $(\beta,h) \in \calU$. 
For $h\neq 0$, we know that
$$
\lim_{n\to\infty} \chi_n(\beta,h) = \chi(\beta,h) < \infty.
$$
Since we can rewrite~\eqref{eq-defchin} as
$$
\chi_n = 1- \expec\left[\tanh^2\left(\sqrt{\frac{\beta}{\expec[W_n]}}W_n x_n^* + h\right) \right]+\frac{\frac{\beta}{\expec[W_n]} \expec\left[\left(1-\tanh^2\left(\sqrt{\frac{\beta}{\expec[W_n]}}W_n x_n^* + h\right)\right) W_n \right]^2}{G_n''(x_n^*)}.
$$
it must hold that $G_n''(x_n^*)>0$ for $n$ large enough.
\end{proof}

We can now investigate the moment generating function of the normalized weighted spin sum.

\begin{proposition}\label{prop-cltweighted}
Suppose that Condition~\ref{cond-WeightReg}(i)--(ii) holds. Then,
$$
\lim_{n\to\infty} \expec\biggl[\exp\biggl\{s\Bigl(\sqrt{\frac{\beta}{\expec[W_n]}} \frac{1}{\sqrt{n}}\sum_{i\in[n]}w_i\sigma_i - \sqrt{n}x_n^*\Bigr)\biggr\}\biggr] = e^{c''(0) \frac{s^2}{2}}, 
$$
where
$$
c(s) = \lim_{n\to\infty} c_n(s),
$$
and $c_n$ is defined in Lemma \ref{lem-cmf-weigthedsum}. 
In particular,
$$
\sqrt{\frac{\beta}{\expec[W_n]}} \frac{1}{\sqrt{n}}\sum_{i\in[n]}w_i\sigma_i - \sqrt{n}x_n^* \stackrel{d}{\longrightarrow} \mathcal{N}\left(0,c''(0) \right), 
$$
and all moments of the l.h.s.~converge to that of this normal distribution.
\end{proposition}

\begin{proof}
Note that, with $s_n = \frac{s}{\sqrt{n}}$,
\begin{align*}
\log \expec\biggl[\exp\biggl\{s\Bigl(\sqrt{\frac{\beta}{\expec[W_n]}}  \frac{1}{\sqrt{n}}\sum_{i\in[n]}w_i\sigma_i &- \sqrt{n}x_n^*\Bigr)\biggr\}\biggr] = n c_n(s_n) - n s_n x_n^*\nn\\
&=n\left(c_n(0) + (c_n'(0)- x_n^*)s_n + c_n''(s_n^*)\frac{s_n^2}{2}\right)\nn\\
&=nc_n(0) +\sqrt{n} (c_n'(0)- x_n^*)s + c_n''(s_n^*)\frac{s^2}{2},
\end{align*}
for some $s_n^* \in (0,s_n)$. Clearly $c_n(0)=0$.

As mentioned, the numerator of~\eqref{eq-cmf-weigthedsum} can be interpreted as the partition function of an Ising model, and hence $c_n(s)$ is the difference of two pressures. Hence, the convergence of $c_n(s)$ can be proved as in~\cite[Sec.~2.1]{GiaGibHofPri16}. Moreover, this means that the monotonicity and convexity properties of the Ising model can be used to show that
$$
\lim_{n\to\infty} c_n''(s_n) = c''(0),
$$ 
see~\cite[Sec.~2.3]{GiaGibHofPri15} for details.

It remains to show that $\sqrt{n} (c_n'(0)- x_n^*) = o(1)$. For this, we use Lemma~\ref{lem-cmf-weigthedsum} with $a=x_n^*$ and
$$
\frac{\dint}{\dint s} G_n(x;s) = G_n'(x+s)-(x+s),
$$
to obtain that
\begin{align*}
\sqrt{n} (c_n'(0)- x_n^*)  &= \sqrt{n}\frac{\int_{-\infty}^\infty \left(\frac{x}{\sqrt{n}}+x_n^*-G'_n(\frac{x}{\sqrt{n}}+x_n^*)\right)e^{-n G_n\left(\frac{x}{\sqrt{n}}+x_n^*\right)} \dint x}{\int_{-\infty}^\infty e^{-n G_n\left(\frac{x}{\sqrt{n}}+x_n^*\right)} \dint x} - \sqrt{n}x_n^* \nn\\
&=\frac{\int_{-\infty}^\infty \left(x-\sqrt{n}G'_n(\frac{x}{\sqrt{n}}+x_n^*)\right)e^{-n G_n\left(\frac{x}{\sqrt{n}}+x_n^*\right)} \dint x}{\int_{-\infty}^\infty e^{-n G_n\left(\frac{x}{\sqrt{n}}+x_n^*\right)} \dint x}.
\end{align*}
Taylor expanding $G'_n(\frac{x}{\sqrt{n}}+x_n^*)$ and $G_n(\frac{x}{\sqrt{n}}+x_n^*)$ around $x_n^*$ gives
\begin{align*}
\sqrt{n} (c_n'(0)- x_n^*)  &= \frac{\int_{-\infty}^\infty \left(x-\sqrt{n}G'_n(x_n^*) - G_n''(x_n^*)x + \mathcal{O}(1/\sqrt{n})\right)e^{-n G_n\left(x_n^*\right)-\sqrt{n}G_n'(x_n^*) x - G_n''(x_n^*)\frac{x^2}{2}+ \mathcal{O}(1/\sqrt{n})} \dint x}{\int_{-\infty}^\infty e^{-n G_n\left(x_n^*\right)-\sqrt{n}G_n'(x_n^*) x - G_n''(x_n^*)\frac{x^2}{2}+ \mathcal{O}(1/\sqrt{n})} \dint x} \nn\\
&=(1-G_n''(x_n^*))\frac{\int_{-\infty}^\infty xe^{ - G_n''(x_n^*)\frac{x^2}{2}+ \mathcal{O}(1/\sqrt{n})} \dint x}{\int_{-\infty}^\infty e^{- G_n''(x_n^*)\frac{x^2}{2}+ \mathcal{O}(1/\sqrt{n})} \dint x} +\mathcal{O}(1/\sqrt{n})=\mathcal{O}(1/\sqrt{n}),
\end{align*}
where we used that $G_n'(x_n^*)=0$ and that in the limit $n\to\infty$ the integral in the numerator equals $0$ since this is an integral over an odd function.
\end{proof}

In the above proposition, we use an explicit centering. Instead, we can also center with the expectation. In that case, we can prove a similar result also when the spins are weighted by different quantities as we show now.

\begin{lemma}\label{lem-clt-tweighted}
Suppose that $(t_i)_{i\in[n]}$ is a sequence satisfying Condition~\ref{cond-WeightReg}(i)--(ii) with $W_n$ replaced by $T_n:=t_I$, with $I\sim Uni[n]$. Let 
$$
\tilde{c}_n(s) = \frac1n \log \expec\biggl[\exp\biggl(s \sum_{i\in[n]} t_i\sigma_i\biggr)\biggr].
$$
Then,
$$
\lim_{n\to\infty} \expec\biggl[\exp\biggl\{s\Bigl(\frac{1}{\sqrt{n}}\sum_{i\in[n]}t_i(\sigma_i - \expec[\sigma_i])\Bigr)\biggr\}\biggr] = e^{\tilde{c}''(0) \frac{s^2}{2}},
$$
where
$$
\tilde{c}(s) = \lim_{n\to\infty} \tilde{c}_n(s).
$$
In particular,
$$
\frac{1}{\sqrt{n}}\sum_{i\in[n]}t_i(\sigma_i - \expec[\sigma_i]) \stackrel{d}{\longrightarrow} \mathcal{N}\left(0,\tilde{c}''(0)\right),
$$
and all moments of the l.h.s.\ converge to that of this normal distribution.
\end{lemma}

\begin{proof}
We proceed as in the previous proposition. We write,with $s_n = \frac{s}{\sqrt{n}}$,
$$
\log \expec\biggl[\exp\biggl\{s\Bigl(\frac{1}{\sqrt{n}}\sum_{i\in[n]}t_i(\sigma_i - \expec[\sigma_i])\Bigr)\biggr\}\biggr] =n\tilde{c}_n(0) +\sqrt{n} \biggl(\tilde{c}_n'(0)- \frac{1}{n}\sum_{i\in[n]}t_i \expec[\sigma_i]\biggr)s + \tilde{c}_n''(s_n^*)\frac{s^2}{2},
$$
for some $s_n^* \in (0,s_n)$. Again, $\tilde{c}_n(0)=0$. Since $\tilde{c}_n(s)$ is a cumulant generating function, 
$$
\tilde{c}'_n(0) = \frac1n \expec\biggl[\sum_{i\in[n]}t_i \sigma_i\biggr],
$$
so that $\tilde{c}_n'(0)- \frac{1}{n}\sum_{i\in[n]}t_i \expec[\sigma_i]=0$. That $\lim_{n\to\infty}\tilde{c}_n''(s_n^*)=\tilde{c}''(0)$ can be shown as above.
\end{proof}

\subsection{Bounds on error terms}\label{sec-errorterms}
We are working in the setting of Lemma \ref{lem-regression}.  Recall that $X=(X_n, \tilde{X}_n)^t$, $X' = (X_n', \tilde{X}_n')^t$ and $\lambda = \frac 1n$.
Note that the inverse of the matrix $\Lambda$ in~\eqref{eq-regressionICW} is given by
$$
\Lambda^{-1} = \begin{pmatrix}1 & c \, \sigma^2(x_n^*,\beta,h) \\ 0 & \sigma^2(x_n^*,\beta,h) \end{pmatrix},
$$
so that with the notations of Theorem \ref{thm-MarginalStein} we obtain
$$
\ell D D_1 = (X_n-X'_n)^2 +c \, \sigma^2(x_n^2,\beta,h) \,  (X_n-X'_n)(\tilde{X}_n-\tilde{X}'_n),
$$
To prove our main result, we apply Theorem~\ref{thm-MarginalStein}. We first bound the first term of~\eqref{eq-thmStein}:

\begin{lemma}\label{lem-bound1sttermerrors}
We have the following bound:
\begin{eqnarray*}
\expec\left[\left|1-\frac{1}{2\lambda} \expec\left[(X_n-X'_n)^2 +c\, \sigma^2(x_n^*,\beta,h) \, (X_n-X'_n)(\tilde{X}_n-\tilde{X}'_n) \,\big|\, \mathcal{F}_n\right] \right|\right] 
& & \\ \leq \expec[|R_3+R_4+R_5+\hat{R}_3+\hat{R}_4+\hat{R}_5|],
\end{eqnarray*}
where
\begin{align*}
R_3&=\frac{1}{\chi_n}\frac{1}{n} \sum_{i\in[n]}\sigma_i \left(\tanh\left(\sqrt{\frac{\beta}{\expec[W_n]}} w_i x_n^*+h\right)-\tanh\left(\frac{\beta w_i}{\expec[W_n]} \tilde{m}_n^i+h\right)\right), \\
R_4&=\frac{1}{\chi_n}\frac{1}{n} \sum_{i\in[n]}\tanh\left(\sqrt{\frac{\beta}{\expec[W_n]}} w_i x_n^*+h\right) \left(\tanh\left(\sqrt{\frac{\beta}{\expec[W_n]}} w_i x_n^*+h\right)-\expec[\sigma_i] \right), \\
R_5&=\frac{1}{\chi_n}\frac{1}{n} \sum_{i\in[n]}\tanh\left(\sqrt{\frac{\beta}{\expec[W_n]}} w_i x_n^*+h\right) \left(\expec[\sigma_i]-\sigma_i \right),\\
\hat{R}_3&=\frac{c\, \sigma^2(x_n^*,\beta,h)}{\sqrt{\chi_n\tilde{\chi}_n}}\frac{1}{n}\sum_{i\in[n]}w_i\sigma_i \left(\tanh\left(\sqrt{\frac{\beta}{\expec[W_n]}} w_i x_n^*+h\right)-\tanh\left(\frac{\beta w_i}{\expec[W_n]} \tilde{m}_n^i+h\right)\right),\\
\hat{R}_4&=\frac{c\, \sigma^2(x_n^*,\beta,h)}{\sqrt{\chi_n\tilde{\chi}_n}}\frac{1}{n} \sum_{i\in[n]}w_i\tanh\left(\sqrt{\frac{\beta}{\expec[W_n]}} w_i x_n^*+h\right) \left(\tanh\left(\sqrt{\frac{\beta}{\expec[W_n]}} w_i x_n^*+h\right)-\expec[\sigma_i] \right),\\
\hat{R}_5&=\frac{c\, \sigma^2(x_n^*,\beta,h)}{\sqrt{\chi_n\tilde{\chi}_n}}\frac{1}{n} \sum_{i\in[n]}w_i\tanh\left(\sqrt{\frac{\beta}{\expec[W_n]}} w_i x_n^*+h\right) \left(\expec[\sigma_i]-\sigma_i \right).
\end{align*}
\end{lemma}

\begin{proof}
Note that
$$
(X_n-X_n')^2 = \frac{1}{n\chi_n} (\sigma_I-\sigma_I')^2 = \frac{2}{n\chi_n}(1-\sigma_I\sigma'_I).
$$
Hence, also using~\eqref{eq-sigmaprimegivenF},
\begin{align*}
\frac{1}{2\lambda}&\expec[(X_n-X_n')^2  \,|\, \calF_n] =\frac{1}{\chi_n}\biggl(1-\frac{1}{n} \sum_{i\in[n]}\sigma_i \expec[\sigma'_i \given \calF_n]\biggr) \nn\\
&=\frac{1}{\chi_n}\biggl(1-\frac{1}{n} \sum_{i\in[n]}\sigma_i \tanh\left(\frac{\beta w_i}{\expec[W_n]} \tilde{m}_n^i+h\right)\biggr) \nn\\
&=\frac{1}{\chi_n}\biggl(1-\expec\left[\tanh^2\left(\sqrt{\frac{\beta}{\expec[W_n]}} W_n x_n^*+h\right)\right]\biggr) \nn\\
& \qquad + \frac{1}{\chi_n}\frac{1}{n} \sum_{i\in[n]}\sigma_i \left(\tanh\left(\sqrt{\frac{\beta}{\expec[W_n]}} w_i x_n^*+h\right)-\tanh\left(\frac{\beta w_i}{\expec[W_n]} \tilde{m}_n^i+h\right)\right) \nn\\
&\qquad + \frac{1}{\chi_n}\frac{1}{n} \sum_{i\in[n]}\tanh\left(\sqrt{\frac{\beta}{\expec[W_n]}} w_i x_n^*+h\right) \left(\tanh\left(\sqrt{\frac{\beta}{\expec[W_n]}} w_i x_n^*+h\right)-\expec[\sigma_i] \right)\nn\\
& \qquad + \frac{1}{\chi_n}\frac{1}{n} \sum_{i\in[n]}\tanh\left(\sqrt{\frac{\beta}{\expec[W_n]}} w_i x_n^*+h\right) \left(\expec[\sigma_i]-\sigma_i \right) \nn\\
&=\frac{1}{\chi_n}\biggl(1-\expec\left[\tanh^2\left(\sqrt{\frac{\beta}{\expec[W_n]}} W_n x_n^*+h\right)\right]\biggr) +R_3+R_4+R_5.
\end{align*}
Since
$$
(X_n-X_n')(\tilde{X}_n-\tilde{X}_n') = \frac{w_I}{n\sqrt{\chi_n\tilde{\chi}_n}} (\sigma_I-\sigma_I')^2 = \frac{2 w_I}{n\sqrt{\chi_n\tilde{\chi}_n}}(1-\sigma_I\sigma'_I),
$$
it can be shown in a similar way, by incorporating the extra factor $w_I$, that
\begin{align*}
\frac{c \, \sigma^2(x_n^*,\beta,h)}{2\lambda} &\expec\left[(X_n-X'_n)(\tilde{X}_n-\tilde{X}'_n) \,\big|\, \mathcal{F}_n\right] \nn\\
&= \frac{c \, \sigma^2(x_n^*,\beta,h)}{\sqrt{\chi_n\tilde{\chi}_n}}\expec\left[ \left(1-\tanh^2\left(\sqrt{\frac{\beta}{\expec[W_n]}} W_n x_n^*+h\right)\right) W_n \right] +\hat{R}_3+\hat{R}_4+\hat{R}_5.
\end{align*}
The lemma follows by observing that
\begin{align*}
\frac{1}{\chi_n}&\biggl(1-\expec\left[\tanh^2\left(\sqrt{\frac{\beta}{\expec[W_n]}} W_n x_n^*+h\right)\right]\biggr)+\frac{c \, \sigma^2(x_n^*,\beta,h)}{\sqrt{\chi_n\tilde{\chi}_n}}\expec\left[ \left(1-\tanh^2\left(\sqrt{\frac{\beta}{\expec[W_n]}} W_n x_n^*+h\right)\right) W_n \right]\nn\\
&=\frac{1}{\chi_n}\biggl(1-\expec\left[\tanh^2\left(\sqrt{\frac{\beta}{\expec[W_n]}} W_n x_n^*+h\right)\right] + \frac{\frac{\beta}{\expec[W_n]} \expec\left[ \left(1-\tanh^2\left(\sqrt{\frac{\beta}{\expec[W_n]}} W_n x_n^*+h\right)\right) W_n \right]^2}{G_n''(x_n^*)}\biggr)\nn\\
&=1,
\end{align*}
which follows from~\eqref{eq-Gn2} and~\eqref{eq-defchin}.
\end{proof}

We bound the second term of~\eqref{eq-thmStein} in a similar way:

\begin{lemma}\label{lem-bound2ndtermerrors}
$$
\frac{1}{\lambda}\expec\left[\left|\expec\left[\left|(X_n-X'_n) +c \, \sigma^2(x_n^*,\beta,h) \, (\tilde{X}_n-\tilde{X}'_n)\right|(X_n-X'_n) \,\big|\, \mathcal{F}_n\right] \right|\right] \leq 2\expec[|\bar{R}_3+\bar{R}_4+\bar{R}_5+\check{R}_3+\check{R}_4+\check{R}_5|],
$$
where
\begin{align*}
\bar{R}_3&=\frac{1}{\chi_n}\frac{1}{n} \sum_{i\in[n]} \left(\tanh\left(\sqrt{\frac{\beta}{\expec[W_n]}} w_i x_n^*+h\right)-\tanh\left(\frac{\beta w_i}{\expec[W_n]} \tilde{m}_n^i+h\right)\right), \\
\bar{R}_4&=\frac{1}{\chi_n}\frac{1}{n} \sum_{i\in[n]} \left(\expec[\sigma_i]-\tanh\left(\sqrt{\frac{\beta}{\expec[W_n]}} w_i x_n^*+h\right) \right), \\
\bar{R}_5&=\frac{1}{\chi_n}\frac{1}{n} \sum_{i\in[n]} \left(\sigma_i -\expec[\sigma_i]\right),\\
\check{R}_3&=\frac{c\, \sigma^2(x_n^*,\beta,h)}{\sqrt{\chi_n\tilde{\chi}_n}}\frac{1}{n}\sum_{i\in[n]}w_i\left(\tanh\left(\sqrt{\frac{\beta}{\expec[W_n]}} w_i x_n^*+h\right)-\tanh\left(\frac{\beta w_i}{\expec[W_n]} \tilde{m}_n^i+h\right)\right),\\
\check{R}_4&=\frac{c\, \sigma^2(x_n^*,\beta,h)}{\sqrt{\chi_n\tilde{\chi}_n}}\frac{1}{n} \sum_{i\in[n]}w_i\left(\expec[\sigma_i]-\tanh\left(\sqrt{\frac{\beta}{\expec[W_n]}} w_i x_n^*+h\right) \right),\\
\check{R}_5&=\frac{c\, \sigma^2(x_n^*,\beta,h)}{\sqrt{\chi_n\tilde{\chi}_n}}\frac{1}{n} \sum_{i\in[n]}w_i\left(\sigma_i -\expec[\sigma_i]\right).
\end{align*}
\end{lemma}

\begin{proof}
We have that
\begin{align*}
|\ell D| D_1 &= |(X_n-X'_n) +c\, \sigma^2(x_n^*,\beta,h) \, (\tilde{X}_n-\tilde{X}'_n)|(X_n-X'_n)\nn\\
&=\frac1n\left[\frac{1}{\chi_n}+\frac{c \, \sigma^2(x_n^*,\beta,h)}{\sqrt{\chi_n \tilde\chi_n}}w_I\right] | \sigma_I-\sigma'_I | (\sigma_I-\sigma'_I) \nn\\
& = \frac2n\left[\frac{1}{\chi_n}+\frac{c\, \sigma^2(x_n^*,\beta,h)}{\sqrt{\chi_n \tilde\chi_n}}w_I\right] (\sigma_I-\sigma'_I),
\end{align*}
where the last equality follows, since $| \sigma_I-\sigma'_I | $ can only take values $2$ or $0$.

\noindent
Hence with $\tilde{m}_n^i$ given by \eqref{mni} we obtain

\begin{align*}
\frac{1}{\lambda} \expec\left[ |\ell D| D_1 \given \calF_n\right] &=  \frac{2}{n} \sum_{i=1}^n \left[\frac{1}{\chi_n}+\frac{c \, \sigma^2(x_n^*,\beta,h)}{\sqrt{\chi_n \tilde\chi_n}}w_i\right] \left[\sigma_i - \tanh\left(\frac{\beta w_i}{\expec[W_n]} \tilde{m}_n^i+h\right)\right] \\
& \hspace{-4cm} = \frac{2}{n} \sum_{i=1}^n \left[\frac{1}{\chi_n}+\frac{c\, \sigma^2(x_n^*,\beta,h)}{\sqrt{\chi_n \tilde\chi_n}}w_i\right] \biggl[\left( \tanh\left(\sqrt{\frac{\beta}{\expec[W_n]}} w_i x_n^*+h\right)- \tanh \left( \frac{\beta w_i}{\expec[W_n]} \tilde{m}_n^i+h \right) \right) \\
& \qquad + \left( \expec[\sigma_i]- \tanh \left( \sqrt{\frac{\beta}{\expec[W_n]}} w_i x_n^*+h \right) \right) + \left(\sigma_i-\expec[\sigma_i]\right)\biggr].
\end{align*}
Expanding out both square brackets gives the six error terms of the lemma.
\end{proof}

The error terms can be bounded as follows.

\begin{lemma}\label{lem-errorterms}
\begin{align*}
\expec[|R_1|]&\leq  \frac{\beta}{\sqrt{\chi_n}}\frac{\expec[W_n^2]}{\expec[W_n]} \frac{1}{\sqrt{n}}, \\
\expec[|R_2|]&\leq\frac{\tilde{\chi}_n}{\sqrt{\chi_n}} \left(\frac{\beta}{\expec[W_n]}\right)^2 \expec[W_n^2] \expec[\tilde{X}_n^2] \frac{1}{\sqrt{n}}, \\
\expec[|R_3|],\expec[|\bar{R}_3|],\expec[|R_4|],\expec[|\bar{R}_4|]&\leq \frac{\beta \sqrt{\tilde{\chi}_n}}{\chi_n} \expec[|\tilde{X}_n|] \frac{1}{\sqrt{n}} + \frac{\beta}{\chi_n}\frac{\expec[W_n^2]}{\expec[W_n]} \frac{1}{n}, \\
\expec[|\tilde{R}_1|] &\leq   \frac{\beta}{\sqrt{\tilde{\chi}_n}} \frac{\expec[W_n^3]}{\expec[W_n]} \frac{1}{\sqrt{n}}, \\
\expec[|\tilde{R}_2|] &\leq 2\sqrt{\tilde{\chi}_n}\left(\frac{\beta}{\expec[W_n]}\right)^2 \expec[W_n^3]\expec[\tilde{X}_n^2] \frac{1}{\sqrt{n}}, \\
\expec[|\hat{R}_3|],\expec[|\check{R}_3|],\expec[|\hat{R}_4|],\expec[|\check{R}_4|]&\leq\frac{\beta c \, \sigma^2(x_n^*,\beta,h)}{\sqrt{\chi_n}} \frac{\expec[W_n^2]}{\expec[W_n]} \expec[|\tilde{X}_n|] \frac{1}{\sqrt{n}}+\frac{\beta c \, \sigma^2(x_n^*,\beta,h)}{\sqrt{\chi_n\tilde{\chi}_n}}\frac{\expec[W_n^3]}{\expec[W_n]} \frac{1}{n}. 
\end{align*}
\end{lemma}

\begin{proof}
Since $\tanh$ is $1$-Lipschitz, $\tilde{m}_n-\tilde{m}_n^i = w_i\sigma_i/n$ and $|\sigma_i|=1$,
$$
\left|\tanh\left(\frac{\beta w_i}{\expec[W_n]} \tilde{m}_n+h\right) - \tanh\left(\frac{\beta w_i}{\expec[W_n]} \tilde{m}_n^i+h\right)\right| \leq \frac{\beta w_i^2}{n\expec[W_n]}.
$$
From this, the bounds on $R_1$ and $\tilde{R}_1$  follow. Using that $\tanh$ is $1$-Lipschitz, it also follows with \eqref{eq-deftildeXn} that
\begin{align*}
\left|\tanh\left(\sqrt{\frac{\beta}{\expec[W_n]}} w_i x_n^*+h\right)-\tanh\left(\frac{\beta w_i}{\expec[W_n]} \tilde{m}_n^i+h\right)\right| & \leq w_i \left|\sqrt{\frac{\beta}{\expec[W_n]}}  x_n^*-\frac{\beta}{\expec[W_n]} \tilde{m}_n^i\right| \nn\\
= w_i \frac{\beta}{\expec[W_n]} \left|\tilde{M}_n- \tilde{m}_n+w_i\frac{\sigma_i}{n}\right|& \leq w_i \frac{\beta \sqrt{\tilde{\chi}_n}}{\expec[W_n]} |\tilde{X}_n| \frac{1}{\sqrt{n}}+\frac{\beta w_i^2}{n\expec[W_n]}.
\end{align*}
From this, the bounds on $R_3, \bar{R}_3, \hat{R}_3$ and $\check{R}_3$ follow. Observe that, by~\eqref{eq-sigmaprimegivenF},
$$
\expec[\sigma_i] = \expec\left[\expec[\sigma_i \,|\, \calF_n^i]\right] = \expec\left[\tanh\left(\frac{\beta w_i}{\expec[W_n]} \tilde{m}_n^i+h\right)\right],
$$
and $|\tanh(x)| \leq 1$, so that the bounds for $R_3$ and $\hat{R}_3$ also hold for $R_4,\bar{R}_4$ and $\hat{R}_4,\check{R}_4$, respectively.

To bound $R_2$, we use the Taylor expansion 
$$
\tanh(x) = \tanh(a)+(1-\tanh^2(a))(x-a)-\tanh(\xi)(1-\tanh^2(\xi)) (x-a)^2,
$$
for some $\xi$ between $x$ and $a$. From the computations in~\eqref{eq-TaylorR2} it follows that
\begin{align*}
|R_2| & \leq \left|\frac{\sqrt{n}}{\sqrt{\chi_n}} \frac{1}{n} \sum_{i\in[n]} \tanh(\xi_i)(1-\tanh^2(\xi_i))\left(\frac{\beta w_i}{\expec[W_n]}\tilde{m}_n-\sqrt{\frac{\beta}{\expec[W_n]}} w_i x_n^*\right)^2 \right|\nn\\
&= \left|\frac{\sqrt{n}}{\sqrt{\chi_n}} \frac{1}{n} \sum_{i\in[n]} \tanh(\xi_i)(1-\tanh^2(\xi_i))\left(\frac{\beta w_i}{\expec[W_n]}\right)^2\frac{\tilde{\chi}_n}{n}\tilde{X}_n^2 \right| \nn\\
&\leq \frac{\tilde{\chi}_n}{\sqrt{\chi_n}} \left(\frac{\beta}{\expec[W_n]}\right)^2 \expec[W_n^2] \tilde{X}_n^2 \frac{1}{\sqrt{n}},
\end{align*}
where we used that $|\tanh(x)|\leq1$.

To bound $\tilde{R}_2$, we expand $G_n'\left(\sqrt{\frac{\beta}{\expec[W_n]}} \tilde{m}_n\right)$ around $x_n^*$ and use that $G_n'(x_n^*)=0$ by definition of $x_n^*$,
and use \eqref{eq-deftildeXn} and \eqref{eq-deflambdasigma} to obtain that, for some $\xi$ between $\sqrt{\frac{\beta}{\expec[W_n]}} \tilde{m}_n$ and $x_n^*$,
\begin{align*}
\tilde{R}_2&=\frac{\sqrt{n}}{\sqrt{\tilde{\chi}_n}} \sqrt{\frac{\expec[W_n]}{\beta}} G'_n\left(\sqrt{\frac{\beta}{\expec[W_n]}} \tilde{m}_n\right) - \frac{1}{\sigma^2(x_n^*,\beta,h)}\tilde{X}_n \nn\\
&=\frac{\sqrt{n}}{\sqrt{\tilde{\chi}_n}} \sqrt{\frac{\expec[W_n]}{\beta}} G_n''(x_n^*) \left(\sqrt{\frac{\beta}{\expec[W_n]}} \tilde{m}_n-x_n^*\right)-G_n''(x_n^*)\tilde{X}_n\nn\\
&\qquad +\frac{\sqrt{n}}{\sqrt{\tilde{\chi}_n}} \sqrt{\frac{\expec[W_n]}{\beta}} G_n'''(\xi) \left(\sqrt{\frac{\beta}{\expec[W_n]}} \tilde{m}_n-x_n^*\right)^2 \nn\\
&= \frac{1}{\sqrt{n}}\sqrt{\tilde{\chi}_n}\sqrt{\frac{\beta}{\expec[W_n]}} \tilde{X}_n^2 G_n'''(\xi).
\end{align*}
Differentiating~\eqref{eq-Gn2} gives
$$
G'''_n(\xi) = 2\left(\frac{\beta}{\expec[W_n]}\right)^{3/2}\expec\left[\tanh\left(\sqrt{\frac{\beta}{\expec[W_n]}}W_n \xi + h\right)\left(1-\tanh^2\left(\sqrt{\frac{\beta}{\expec[W_n]}}W_n \xi + h\right)\right) W_n^3 \right].
$$
Since $|\tanh x|\leq1$, we obtain
$$
|G'''_n(\xi)| \leq 2\left(\frac{\beta}{\expec[W_n]}\right)^{3/2} \expec[W_n^3],
$$
from which the bound on $\tilde{R}_2$ follows.
\end{proof}

\noindent
We now combine all results to prove our main result.

\begin{proof}[Proof of Theorem~\ref{thm-berryesseen-magnetization}]
We apply Theorem~\ref{thm-MarginalStein}. To show that the first term of~\eqref{eq-thmStein} is $\mathcal{O}(1/\sqrt{n})$  it suffices to show, by Lemma~\ref{lem-bound1sttermerrors}, that 
$$
\expec[|R_3|],\expec[|R_4|],\expec[|R_5|], \expec[|\hat{R}_3|],\expec[|\hat{R}_4|],\expec[|\hat{R}_5|] \leq \frac{C}{\sqrt{n}},
$$
where $C$ is a constant not depending on $n$ that may change from line to line.

For the second term of~\eqref{eq-thmStein}, it suffices to show, by Lemma~\ref{lem-bound2ndtermerrors}, that 
$$
\expec[|\bar{R}_3|],\expec[|\bar{R}_4|],\expec[|\bar{R}_5|], \expec[|\check{R}_3|],\expec[|\check{R}_4|],\expec[|\check{R}_5|] \leq \frac{C}{\sqrt{n}}.
$$

Note that it follows from Proposition~\ref{prop-cltweighted} that $\expec[|\tilde{X}_n|]$ is uniformly bounded. By Condition~\ref{cond-WeightReg}(i)--(iii), also the first three moments of $W_n$ are uniformly bounded. From this and Lemma~\ref{lem-errorterms}, the bounds on $\expec[|R_3|], \expec[|\bar{R}_3|],\expec[|\hat{R}_3|], \expec[|\check{R}_3|], \expec[|R_4|], \expec[|\bar{R}_4|],\expec[|\hat{R}_4|]$ and $\expec[|\check{R}_4|]$ follow. Remark that this is one of the places where we see that
the constant in our Berry-Esseen bound depends on $(w_i)_{i \geq 1}$.

By rewriting
$$
\expec[|R_5|]=\frac{1}{\chi_n}\expec\left[\left|\frac{1}{\sqrt{n}} \sum_{i\in[n]}\tanh\left(\sqrt{\frac{\beta}{\expec[W_n]}} w_i x_n^*+h\right) \left(\sigma_i-\expec[\sigma_i] \right)\right|\right]\frac{1}{\sqrt{n}},
$$
it can be seen that $\expec[|R_5|]$ is of the form considered in Lemma~\ref{lem-clt-tweighted} with $t_i=\tanh\left(\sqrt{\frac{\beta}{\expec[W_n]}} w_i x_n^*+h\right)$. Hence, it follows from Lemma~\ref{lem-clt-tweighted} that $\sqrt{n}\expec[|R_5|]$ is uniformly bounded. A similar argument holds for $\bar{R}_5, \hat{R}_5$ and $\check{R}_5$.

For the third term in~\eqref{eq-thmStein}, note that by~\eqref{eq-regressionICW} that $R = \begin{pmatrix} R_1 +R_2 \\ \tilde{R_1}+ \tilde{R_2}\end{pmatrix}$ and we have that
$$
\expec\left[|\ell R|\right]  = \expec\left[|R_1+R_2+c\sigma^2(x_n^*,\beta,h) (\tilde{R}_1+\tilde{R}_2)|\right],
$$
and it follows from Lemma~\ref{lem-errorterms} and the uniform boundedness of the first three  moments of $W_n$ and all moments of $\tilde{X}_n$ that also this term can be bounded from above by $C/\sqrt{n}$.

If we only assume Condition~\ref{cond-WeightReg}(i)--(ii) and suppose that $\max_{i\in[n]}w_i \geq c\sqrt{n}$ for some $c>0$, then
$$
\expec[W_n^2] = \frac1n \sum_{i\in[n]} w_i^2 \geq c^2+\frac1n \sum_{i: i\neq \argmax w_j} w_i^2 \stackrel{n\to\infty}{\longrightarrow} c^2 + \expec[W^2],
$$
which is in contradiction to Condition~\ref{cond-WeightReg}(ii). Hence, $\max_{i\in[n]}w_i = o(\sqrt{n})$ and also
$$
\expec[W_n^3] = \frac1n \sum_{i\in[n]} w_i^3 \leq \max_{i\in[n]}w_i\ \expec[W_n^2] = o(\sqrt{n}).
$$
This suffices to prove the second statement of Theorem~\ref{thm-berryesseen-magnetization}.
\end{proof}

When one wants to prove the Berry-Esseen bound for the weighted sum of spins, the role of $X_n$ and $\tilde{X}_n$ can be interchanged. In fact, $X_n$ can be ignored in that case and in the first factor of~\eqref{eq-thmStein}, we have to estimate
\begin{align}\label{eq-approxweighted1stterm}
\frac{1}{2\lambda} \expec[\ell D D_1 \given \calF_n] &= \frac{\sigma^2(x_n^*,\beta,h)}{\tilde{\chi}_n} \sum_{i\in[n]}w_i^2\left(1-\sigma_i\expec[\sigma_i'\given \calF_n]\right) \nn\\
&\approx  \frac{\sigma^2(x_n^*,\beta,h)}{\tilde{\chi}_n} \expec\left[\left(1-\tanh^2\left(\sqrt{\frac{\beta}{\expec[W_n]}}W_n x_n^* + h\right)\right)W_n^2\right].
\end{align}
Since we want this to be equal to $1$, we choose $\tilde{\chi}_n$ as in \eqref{eq-choicetildechi}.
The approximation in~\eqref{eq-approxweighted1stterm} can be made precise as in Lemma~\ref{lem-bound1sttermerrors} and the resulting error terms can be shown to be $\calO(1/\sqrt{n})$ as in Lemma~\ref{lem-errorterms} under the assumption of one extra moment in Condition~\ref{cond-WeightReg}. Also the other terms in~\eqref{eq-thmStein} can then be shown to be $\calO(1/\sqrt{n})$ as in Lemmas~\ref{lem-bound2ndtermerrors} and~\ref{lem-errorterms}.

\paragraph*{Acknowledgements.}
A large part of this work was carried out at the Ruhr-Universit\"at Bochum, supported by the Deutsche Forschungsgemeinschaft (DFG) via RTG 2131 \emph{High-dimensional Phenomena in Probability -- Fluctuations and Discontinuity}.

\end{document}